\documentclass{preprint}

\hypersetup{
  pdftitle={Cross product-free matrix pencils for computing generalized
    singular values},
  pdfauthor={Ian N. Zwaan},
  pdfsubject={65F15; 15A18},
  pdfkeywords={GSVD; QSVD; RSVD; GEP; augmented matrix; generalized
    singular value decomposition; quotient singular value decomposition;
    restricted singular value decomposition; generalized eigenvalue
    problem}
  pdflang={en-US},
  unicode=true,
}

\usepackage[USenglish]{babel}
\PassOptionsToPackage{USenglish}{babel}
\usepackage[space,noadjust,nocompress]{cite}

\usepackage[fleqn,leqno]{amsmath}
\usepackage[fleqn,leqno]{mathtools}
\usepackage{amssymb}
\usepackage{bm}
\renewcommand*\vec\bm

\usepackage{array}
\usepackage{arydshln}

\usepackage{booktabs}
\renewcommand\arraystretch{1.25}

\makeatletter
\newcommand*\centerfloat{%
  \parindent \z@
  \leftskip \z@ \@plus 1fil \@minus \textwidth
  \rightskip\leftskip
  \parfillskip \z@skip}
\makeatother

\usepackage{pgfplots}
\usepackage{pgfplotstable}
\pgfplotscreateplotcyclelist{customlines}{
  solid,loosely dashed,densely dotted,dashdotted
}
\pgfplotsset{
  compat = newest,
  filter discard warning = false,
  legend cell align = left,
  every axis plot/.append style = {
    black,
    mark = none,
    line width = 1pt
  },
  cycle list name = customlines,
  /pgf/number format/1000 sep = {},
  ticklabel style = {font = \footnotesize},
  label style = {font = \footnotesize},
  every axis legend/.append style = {
    font = \footnotesize
  }
}

\DeclareMathOperator{\diag}{diag} 
\DeclareMathOperator{\rank}{rank} 

\newcommand*\minimat[4]{\left[%
  \begin{smallmatrix} #1 & #2 \\ #3 & #4 \end{smallmatrix}%
    \right]}

\makeatletter
\newif\if@borderstar
\def\bordermatrix{\@ifnextchar*{%
\@borderstartrue\@bordermatrix@i}{\@borderstarfalse\@bordermatrix@i*}%
}
\def\@bordermatrix@i*{\@ifnextchar[{\@bordermatrix@ii}{\@bordermatrix@ii[()]}}
\def\@bordermatrix@ii[#1]#2{%
\begingroup
\m@th\@tempdima8.75\p@\setbox\z@\vbox{%
\def\cr{\crcr\noalign{\kern 2\p@\global\let\cr\endline }}%
\ialign {$##$\hfil\kern 2\p@\kern\@tempdima & \thinspace %
\hfil $##$\hfil && \quad\hfil $##$\hfil\crcr\omit\strut %
\hfil\crcr\noalign{\kern -\baselineskip}#2\crcr\omit %
\strut\cr}}%
\setbox\tw@\vbox{\unvcopy\z@\global\setbox\@ne\lastbox}%
\setbox\tw@\hbox{\unhbox\@ne\unskip\global\setbox\@ne\lastbox}%
\setbox\tw@\hbox{%
$\kern\wd\@ne\kern -\@tempdima\left\@firstoftwo#1%
\if@borderstar\kern2pt\else\kern -\wd\@ne\fi%
\global\setbox\@ne\vbox{\box\@ne\if@borderstar\else\kern 2\p@\fi}%
\vcenter{\if@borderstar\else\kern -\ht\@ne\fi%
\unvbox\z@\kern-\if@borderstar2\fi\baselineskip}%
\if@borderstar\kern-2\@tempdima\kern2\p@\else\,\fi\right\@secondoftwo#1 $%
}\null \;\vbox{\kern\ht\@ne\box\tw@}%
\endgroup
}
\makeatother

\pgfplotstableread{
k   sq       aug      new      gsvd
1.0 9.07e-16 6.20e-16 1.39e-15 8.08e-16
1.5 4.30e-15 2.50e-15 1.50e-15 1.00e-15
2.0 3.22e-14 1.84e-14 1.87e-15 1.78e-15
2.5 2.60e-13 1.53e-13 3.40e-15 4.32e-15
3.0 2.28e-12 1.33e-12 8.64e-15 1.20e-14
3.5 2.09e-11 1.21e-11 2.43e-14 3.43e-14
4.0 1.96e-10 1.11e-10 7.12e-14 1.03e-13
4.5 1.86e-09 1.05e-09 2.12e-13 3.07e-13
5.0 1.75e-08 9.93e-09 6.37e-13 9.27e-13
5.5 1.71e-07 9.94e-08 1.94e-12 2.85e-12
6.0 1.67e-06 9.50e-07 6.07e-12 8.93e-12
6.5 1.63e-05 9.14e-06 1.86e-11 2.74e-11
7.0 0.000168 9.26e-05 5.78e-11 8.57e-11
}\qsvdresultsfirst

\pgfplotstableread{
k   sq       aug      new      gsvd
 1.0 9.05e-16 6.17e-16 1.39e-15 8.09e-16
 2.0 5.13e-15 1.09e-15 1.21e-15 6.69e-16
 3.0 2.91e-14 2.72e-15 1.01e-15 5.57e-16
 4.0 1.43e-13 7.78e-15 8.49e-16 5.17e-16
 5.0 6.75e-13 2.29e-14 7.43e-16 4.74e-16
 6.0 2.73e-12 7.18e-14 6.46e-16 4.20e-16
 7.0 1.16e-11 2.16e-13 5.72e-16 3.49e-16
 8.0 4.95e-11 6.80e-13 5.11e-16 3.64e-16
 9.0 1.90e-10 2.11e-12 4.61e-16 3.10e-16
10.0 7.52e-10 6.77e-12 4.23e-16 2.95e-16
11.0 2.77e-09 2.12e-11 3.61e-16 2.75e-16
12.0 9.25e-09 6.81e-11 3.33e-16 2.58e-16
13.0 3.45e-08 2.16e-10 3.07e-16 2.40e-16
}\qsvdresultssecond

\pgfplotstableread{
k   aug      new      rsvd
1.0 6.63e-16 1.46e-15 1.56e-15
1.5 2.51e-15 1.70e-15 1.81e-15
2.0 1.87e-14 2.65e-15 2.84e-15
2.5 1.53e-13 6.14e-15 6.47e-15
3.0 1.32e-12 1.69e-14 1.78e-14
3.5 1.23e-11 4.86e-14 5.26e-14
4.0 1.11e-10 1.45e-13 1.54e-13
4.5 1.07e-09 4.41e-13 4.62e-13
5.0 1.01e-08 1.32e-12 1.40e-12
5.5 9.77e-08 4.08e-12 4.28e-12
6.0 9.43e-07 1.25e-11 1.32e-11
6.5 9.45e-06 3.86e-11 4.14e-11
7.0 9.23e-05 1.21e-10 1.27e-10
}\rsvdresultsfirst

\pgfplotstableread{
k   aug      new      rsvd
 1.0 6.63e-16 1.49e-15 1.54e-15
 2.0 1.13e-15 1.32e-15 1.29e-15
 3.0 2.69e-15 1.12e-15 1.12e-15
 4.0 7.80e-15 9.86e-16 9.67e-16
 5.0 2.35e-14 9.07e-16 9.07e-16
 6.0 7.04e-14 8.11e-16 8.44e-16
 7.0 2.11e-13 7.60e-16 7.62e-16
 8.0 6.80e-13 7.00e-16 7.08e-16
 9.0 2.09e-12 6.63e-16 6.53e-16
10.0 6.76e-12 6.14e-16 6.14e-16
11.0 2.13e-11 5.76e-16 5.78e-16
12.0 6.77e-11 5.34e-16 5.64e-16
13.0 2.10e-10 5.22e-16 5.27e-16
}\rsvdresultssecond

\title{Cross product--free matrix pencils for computing generalized
singular values\thanks{This work was supported in part by the Deutsche
Forschungsgemeinschaft through the collaborative research centre
SFB-TRR55. Version: \today.}}
\author{Ian N. Zwaan\thanks{Faculty of Mathematics and Natural Sciences,
Bergische Universit\"at Wuppertal, \texttt{ianzwaan.com}.}}
\date{}

\shorttitle{Cross product--free matrix pencils}
\shortauthor{I.~N.~Zwaan}

\begin{document}

\maketitle


\begin{abstract}
  It is well known that the generalized (or quotient) singular values of
  a matrix pair $(A, C)$ can be obtained from the generalized
  eigenvalues of a matrix pencil consisting of two augmented matrices.
  The downside of this reformulation is that one of the augmented
  matrices requires a cross products of the form $C^*C$, which may
  affect the accuracy of the computed quotient singular values if $C$
  has a large condition number. A similar statement holds for the
  restricted singular values of a matrix triplet $(A, B, C)$ and the
  additional cross product $BB^*$. This article shows that we can
  reformulate the quotient and restricted singular value problems as
  generalized eigenvalue problems without having to use any
  cross product or any other matrix-matrix product. Numerical
  experiments show that there indeed exist situations in which the new
  reformulation leads to more accurate results than the well-known
  reformulation.
\end{abstract}



\begin{keywords}
  generalized eigenvalue problem, augmented matrix,
  generalized singular value decomposition, GSVD, quotient singular
  value decomposition, QSVD, restricted singular value decomposition,
  RSVD.
\end{keywords}



\begin{AMS}
  65F15; 
  15A18  
\end{AMS}



\section{Introduction}\label{sec:int}

Suppose that $A \in \mathbb{C}^{p\times q}$, then it is well known that
the (ordinary) singular values of $A$ can be obtained from the
eigenvalues of either of the products 
\begin{equation}\label{eq:sqSVD}
  \mathcal A = A^*\!A
  \quad\text{or}\quad
  \mathcal A = AA^*,
\end{equation}
or from the augmented matrix
\begin{equation}\label{eq:augSVD}
  \mathcal A =
  \begin{bmatrix}
    0 & A \\
    A^* & 0
  \end{bmatrix}.
\end{equation}
Likewise, if $C \in \mathbb{C}^{n \times q}$ is a second matrix, then
the generalized singular values (precise definitions of the various
singular value concepts follow later), also called the quotient singular
values \cite{MZ91}, of the matrix pair $(A,C)$ can be obtained either
from the generalized eigenvalues of the pencil
\begin{equation}\label{eq:sqQSVD}
  \mathcal A - \lambda \mathcal B = A^*\!A - \lambda C^*C
\end{equation}
or from the augmented pencil
\begin{equation}\label{eq:augQSVD}
  \mathcal A - \lambda \mathcal B =
  \begin{bmatrix}
    0 & A \\
    A^* & 0
  \end{bmatrix}
  {} - \lambda
  \begin{bmatrix}
    I & 0 \\
    0 & C^*C
  \end{bmatrix}.
\end{equation}
Furthermore, if $B \in \mathbb{C}^{p \times m}$ is a third matrix, then
the restricted singular values of the triplet $(A,B,C)$, can be obtained
from the pencil
\begin{equation}\label{eq:augRSVD}
  \mathcal A - \lambda \mathcal B =
  \begin{bmatrix}
    0 & A \\
    A^* & 0
  \end{bmatrix}
  {} - \lambda
  \begin{bmatrix}
    BB^* & 0 \\
    0 & C^*C
  \end{bmatrix}.
\end{equation}
Except for \eqref{eq:augSVD}, all the above (generalized) eigenvalue
problems require cross products of the form $A^*\!A$, $BB^*$, or $C^*C$.
Textbooks by, for example, Stewart~\cite[Sec.~3.3.2]{Stew01},
Higham~\cite[Sec.~20.4]{High02}, and Golub and
Van~Loan~\cite[Sec.~8.6.3]{GvL13} dictate that these types of products
are undesirable for poorly conditioned matrices, because condition
numbers are squared and accuracy may be lost. This loss of accuracy has
also been investigated, for example, by Jia~\cite{Jia06} for the SVD,
and by and Huang and Jia~\cite{HJ19} for the GSVD.

Numerical methods exist, for large and sparse matrices in particular,
purposefully designed to avoid explicit use of the unwanted cross
products. Examples for the singular value problem include
Golub--Kahan--Lanczos bidiagonalization; see, e.g.,
Demmel~\cite[Sec.~6.3.3]{ET00Bidiag}; and JDSVD by
Hochstenbach~\cite{Hoch01}.  Examples for the quotient singular value
problem include a bidiagonalization method by Zha~\cite{Zha96}, JDGSVD
by Hochstenbach~\cite{Hoch09}, generalized Krylov methods by
Hochstenbach, Reichel, and Yu~\cite{HLY15} and Reichel and
Yu~\cite{RY15matdec,RY15flexarn}, and a Generalized--Davidson based
projection method \cite{ZH17}.

The purpose of this article is to show that we can reformulate the
ordinary, quotient, and restricted singular value problems as a matrix
pencil that consists of two augmented matrices, neither of which require
any cross product, or any other matrix-matrix product.

To see precisely how the eigenvalues and eigenvectors of the matrix
pencils relate to the singular values and vectors, we can use the
Kronecker Canonical Form (KCF). The KCF, detailed in
Section~\ref{sec:kcf}, is the generalization of the Jordan form of a
matrix to matrix pencils, and fully characterises the generalized
eigenstructure of pencils. The next step is to reformulate the ordinary
singular value decomposition (OSVD) in Section~\ref{sec:svd}, and to
analyze the corresponding KCF.  This particular reformulation is purely
for exposition, since the augmented matrix \eqref{eq:augSVD} is already
free of cross products.  That is, the new reformulation of the OSVD is
the simplest case we can consider and the easiest to verify by hand, but
already uses the same general approach we use for the other two singular
value problems. In the next two sections, Section~\ref{sec:qsvd} and
\ref{sec:rsvd}, we discuss the reformulation of the quotient singular
value decomposition (QSVD) and the restricted singular value
decomposition (RSVD). The former is better known and more widely used in
practice, while the latter is more general. The generality of the RSVD
makes it tedious to describe in full detail, and also tedious to get the
KCF of its corresponding cross product--free pencil. Still, its
treatment is essentially identical to the simpler cases of the OSVD and
QSVD.  The numerical experiments that follow in Section~\ref{sec:ne}
show that, for some matrix pairs and triplets, we can compute the
singular values more accurately from the new cross product--free pencils
than from the typical augmented pencils. The new pencils are not without
their own downsides, such as an increased problem size and the presence
of Jordan blocks, which we further discuss in the conclusion in
Section~\ref{sec:con}.

Throughout this text, $I$ denotes an identity matrix of appropriate
size; $M^T$ and $M^*$ denote the transpose and Hermitian transpose,
respectively, of a matrix $M$; and $\otimes$ denotes the Kronecker
product. Furthermore, for some permutation $\pi$ of length $k$, the
corresponding permutation matrix is given by
\begin{equation*}
  \Pi =
  \begin{bmatrix}
    e_{\pi(1)} & e_{\pi(2)} & \dots & e_{\pi(k)}
  \end{bmatrix},
\end{equation*}
where the $e_j$ are the $j$th canonical basis vectors of length $k$.
The notation of permutations and permutation matrices is extended to
block matrices, and permute entire blocks of rows or columns at once.



\section{The Kronecker canonical form}\label{sec:kcf}

As mentioned before, the KCF is a generalization of the Jordan form to
matrix pencils, and its importance is that it fully describes the
generalized eigenvalues and generalized eigenspaces of a matrix pencil.
This also means that two matrix pencils are equivalent if-and-only-if
they have the same KCF \cite[Thm.~5]{Gant2}, where the meaning of
equivalence is as in the following definition.


\begin{definition}[Gantmacher~{\cite[Def.~1]{Gant2}}]
  Two pencils of rectangular matrices $\mathcal A - \lambda \mathcal B$
  and $\mathcal A_1 - \lambda \mathcal B_1$ of the same dimensions $k
  \times \ell$ are called strictly equivalent if there exists
  nonsingular $\mathcal X$ and $\mathcal Y$, independent of $\lambda$,
  such that $\mathcal Y^* (\mathcal A - \lambda \mathcal B) \mathcal X =
  \mathcal A_1 - \lambda \mathcal B_1$.
\end{definition}


Now we are ready for the definition of the KCF, which is given by the
following theorem.


\begin{theorem}[Adapted from K\r{a}gstr\"om~{\cite[Sec.~8.7.2]{ET00KCF}}
  and Gantmacher~{\cite[Ch.~2]{Gant2}}]
  Let $\mathcal A, \mathcal B \in \mathbb{C}^{k \times \ell}$, then
  there exists nonsingular matrices $\mathcal X \in \mathbb{C}^{\ell
  \times \ell}$ and $\mathcal Y \in \mathbb{C}^{k \times k}$ such that
  $\mathcal Y^* (\mathcal A - \lambda \mathcal B) \mathcal X$ equals
  \begin{equation*}
    \operatorname{diag}(
      0_{\beta_0 \times \alpha_0}, 
      L_{\alpha_1}, \dots, L_{\alpha_m},
      L_{\beta_1}^T, \dots, L_{\beta_n}^T,
      N_{\gamma_1}, \dots, N_{\gamma_p},
      J_{\delta_1}(\zeta_1), \dots, J_{\delta_q}(\zeta_q)
    );
  \end{equation*}
  where the
  \begin{equation*}
    L_{\alpha_j} =
    \begin{bmatrix}
      0 & 1 \\
      & \ddots & \ddots \\
      && 0 & 1
    \end{bmatrix}
    {} - \lambda
    \begin{bmatrix}
      1 & 0 \\
      & \ddots & \ddots \\
      && 1 & 0
    \end{bmatrix}
  \end{equation*}
  are $\alpha_j \times (\alpha_j + 1)$ singular blocks of right (or
  column) minimal index $\alpha_j$, the
  \begin{equation*}
    L_{\beta_j}^T =
    \begin{bmatrix}
      0 \\
      1 & \ddots \\
      & \ddots & 0 \\
      && 1
    \end{bmatrix}
    {} - \lambda
    \begin{bmatrix}
      1 \\
      0 & \ddots \\
      & \ddots & 1 \\
      && 0
    \end{bmatrix}
  \end{equation*}
  are $(\beta_j + 1) \times \beta_j$ singular blocks of left (or row)
  minimal index $\beta_j$, the
  \begin{equation*}
    N_{\gamma_j} =
    \begin{bmatrix}
      1 & 0 \\
      & \ddots & \ddots \\
      && \ddots & 0 \\
      &&& 1
    \end{bmatrix}
    {} - \lambda
    \begin{bmatrix}
      0 & 1 \\
      & \ddots & \ddots \\
      && \ddots & 1 \\
      &&& 0
    \end{bmatrix}
  \end{equation*}
  are $\gamma_j \times \gamma_j$ Jordan blocks corresponding to an
  infinite eigenvalues $\gamma_j$, and the
  \begin{equation*}
    J_{\delta_j}(\zeta_j) =
    \begin{bmatrix}
      \zeta & 1 \\
      & \ddots & \ddots \\
      && \ddots & 1 \\
      &&& \zeta
    \end{bmatrix}
    {} - \lambda
    \begin{bmatrix}
      1 & 0 \\
      & \ddots & \ddots \\
      && \ddots & 0 \\
      &&& 1
    \end{bmatrix}
  \end{equation*}
  are $\delta_j \times \delta_j$ Jordan blocks corresponding to finite
  eigenvalues $\zeta_j \in \mathbb{C}$.
\end{theorem}


The meaning of minimal indices is not important for this paper, and an
interested reader may refer to the references for more information.
Furthermore, the zero block $0_{\beta_0 \times \alpha_0}$ corresponds to
$\alpha_0$ blocks $L_0$ and $\beta_0$ blocks $L_0^T$ by convention. We
include it here explicitly for clarity, and because in this paper we
have no KCF with any other $L_\alpha$ and $L_\beta^T$ blocks.
Moreover, in this paper all Jordan blocks corresponding to finite
nonzero eigenvalues $\zeta_j$ will be such that $J_{\delta_j}(\zeta_j) =
\zeta_j - \lambda$; that is, $\delta_j = 1$ whenever $\zeta_j \neq 0$.




\section{The ordinary singular value decomposition}\label{sec:svd}

The augmented matrix \eqref{eq:augSVD} does not contain any cross
products. Still, we can take the cross product--free matrix pencil for
the QSVD and consider the OSVD as a special case. The main reason for us
to do so, is to show how we can determine the corresponding KCF for the
simplest case that we can consider. That is, this section sets the stage
and the reformulations for the QSVD and RSVD and their proofs in the
sections to come, follow from the same general idea.

Let us start by recalling the definition of the ordinary singular value
decomposition with the following theorem (see, e.g., Golub and
Van~Loan~\cite{GvL13}).


\begin{theorem}[Ordinary singular value decomposition (OSVD)]\label{thm:svd}
  Let $A \in \mathbb{C}^{p \times q}$; then there exist unitary matrices
  $U \in \mathbb{C}^{p \times p}$ and $V \in \mathbb{C}^{q \times q}$
  such that
  \begin{equation*}
    \Sigma = U^*\!AV =\;
    \bordermatrix[{[]}]{%
        & q_1      & q_2 \cr
    p_1 & D_\sigma & 0   \cr
    p_2 & 0        & 0
    },
  \end{equation*}
  where $p = p_1 + p_2$, $q = q_1 + q_2$, and $p_1 = q_1$. Furthermore,
  $D_\sigma = \diag(\sigma_1, \dots, \sigma_{p_1})$ with $\sigma_j > 0$
  for all $j = 1$, $2$, \dots, $p_1$.
\end{theorem}


Next, let us consider the eigenvalue decomposition of a $4 \times 4$
pencil that we can consider as the cross product--free pencil for the $1
\times 1$ matrix $\sigma \ge 0$.


\begin{lemma}\label{thm:svdfour}
  Suppose $\sigma$ is a positive real number, and consider the pencil
  \begin{equation}\label{eq:svdfour}
    \mathcal A - \lambda \mathcal B =
    \begin{bmatrix}
      0 & \sigma & 0 & 0 \\
      \sigma & 0 & 0 & 0 \\
      0 & 0 & 1 & 0 \\
      0 & 0 & 0 & 1
    \end{bmatrix}
    {} - \lambda
    \begin{bmatrix}
      0 & 0 & 1 & 0 \\
      0 & 0 & 0 & 1 \\
      1 & 0 & 0 & 0 \\
      0 & 1 & 0 & 0
    \end{bmatrix}.
  \end{equation}
  Then the unitary matrices
  \begin{equation*}
    \mathcal X = \frac{1}{2}
    \begin{bmatrix}
      1 & -1 & -i & \phantom{-}i \\
      1 & -1 & \phantom{-}i & -i \\
      1 & \phantom{-}1 & \phantom{-}1 & \phantom{-}1 \\
      1 & \phantom{-}1 & -1 & -1
    \end{bmatrix}
    \quad\text{and}\quad
    \mathcal Y = \frac{1}{2}
    \begin{bmatrix}
      1 & \phantom{-}1 & \phantom{-}1 & \phantom{-}1 \\
      1 & \phantom{-}1 & -1 & -1 \\
      1 & -1 & -i & \phantom{-}i \\
      1 & -1 & \phantom{-}i & -i
    \end{bmatrix}
  \end{equation*}
  are such that
  \begin{equation*}
    \mathcal Y^* (\mathcal A - \lambda \mathcal B) \mathcal X =
    \operatorname{diag}(\sqrt\sigma, -\sqrt\sigma, i\sqrt\sigma, -i\sqrt\sigma)
    {} - \lambda I.
  \end{equation*}
\end{lemma}
\begin{proof}
  The proof is by direct verification.
\end{proof}


With the above definition and lemma, we can state and prove the
following theorem and corollary.


\begin{theorem}\label{thm:sqfreeSVD}
  Let $A$ and its corresponding SVD be as in Theorem~\ref{thm:svd}. Then
  the KCF of the pencil
  \begin{equation}\label{eq:sqfreeSVD}
    \mathcal A - \lambda \mathcal B =
    \begin{bmatrix}
      0   & A & 0 & 0 \\
      A^* & 0 & 0 & 0 \\
      0   & 0 & I & 0 \\
      0   & 0 & 0 & I
    \end{bmatrix}
    {} - \lambda
    \begin{bmatrix}
      0   & 0 & I & 0 \\
      0   & 0 & 0 & I \\
      I   & 0 & 0 & 0 \\
      0   & I & 0 & 0
    \end{bmatrix}
  \end{equation}
  consists of the following blocks.
  \begin{enumerate}
    \item A series of $p_2 + q_2$ blocks $J_2(0)$.
    \item The blocks
      $J_1(\sqrt{\sigma_1})$, \dots, $J_1(\sqrt{\sigma_{p_1}})$, 
      $J_1(-\sqrt{\sigma_1})$, \dots, $J_1(-\sqrt{\sigma_{p_1}})$,
      $J_1(i\sqrt{\sigma_1})$, \dots, $J_1(i\sqrt{\sigma_{p_1}})$,
      $J_1(-i\sqrt{\sigma_1})$, \dots, $J_1(-i\sqrt{\sigma_{p_1}})$,
      where $\sigma_1$, \dots, $\sigma_{p_1}$ are the nonzero singular
      values of $A$.
  \end{enumerate}
\end{theorem}
\begin{proof}
  Let $\mathcal A_0 - \lambda \mathcal B_0 = \mathcal A - \lambda
  \mathcal B$, and define the transformations $\mathcal X_0 = \mathcal
  Y_0 = \diag(U, V, U, V)$. Then the pencil $\mathcal A_1 - \lambda
  \mathcal B_1 = \mathcal Y_0^* (\mathcal A_0 - \lambda \mathcal B_0)
  \mathcal X_0$ is a square $8 \times 8$ block matrix of dimension
  \begin{equation*}
    \underbrace{p_1 + p_2}_{p}
    {} + \underbrace{q_1 + q_2}_{q}
    {} + \underbrace{p_1 + p_2}_{p}
    {} + \underbrace{q_1 + q_2}_{q}.
  \end{equation*}
  Now let $\mathcal X_1$ and $\mathcal Y_1$ be the permutation matrices
  corresponding to the permutations
  \begin{equation*}
    \pi_{\mathcal X} = (
      2, 6, 4, 8,\quad
      1, 3, 5, 7
    )
    \quad\text{and}\quad
    \pi_{\mathcal Y} = (
      6, 2, 8, 4,\quad
      1, 3, 5, 7
    ),
  \end{equation*}
  respectively. Then the pencil $\mathcal A_2 - \lambda \mathcal B_2 =
  \mathcal Y_1^* (\mathcal A_1 - \lambda \mathcal B_1) \mathcal X_1$ is
  block diagonal, and has the following blocks along its diagonal.
  \begin{enumerate}
    \item The $(p_2 + p_2 + q_2 + q_2) \times (p_2 + p_2 + q_2 + q_2)$
      block
      \begin{equation*}
        \begin{bmatrix}
          0 & I & 0 & 0 \\
          0 & 0 & 0 & 0 \\
          0 & 0 & 0 & I \\
          0 & 0 & 0 & 0
        \end{bmatrix}
        {} - \lambda
        \begin{bmatrix}
          I & 0 & 0 & 0 \\
          0 & I & 0 & 0 \\
          0 & 0 & I & 0 \\
          0 & 0 & 0 & I
        \end{bmatrix}
        =
        \begin{bmatrix}
          J_2(0) \otimes I \\
          & J_2(0) \otimes I
        \end{bmatrix},
      \end{equation*}
      which yields $p_2 + q_2$ blocks $J_2(0)$ in the KCF of
      \eqref{eq:sqfreeSVD} after suitable permutations.

    \item The $(p_1 + q_1 + p_1 + q_1) \times (p_1 + q_1 + p_1 + q_1)$
      block
      \begin{equation*}
        \begin{bmatrix}
          0 & D_\sigma & 0 & 0 \\
          D_\sigma & 0 & 0 & 0 \\
          0 & 0 & I & 0 \\
          0 & 0 & 0 & I
        \end{bmatrix}
        {} - \lambda
        \begin{bmatrix}
          0 & 0 & I & 0 \\
          0 & 0 & 0 & I \\
          I & 0 & 0 & 0 \\
          0 & I & 0 & 0
        \end{bmatrix},
      \end{equation*}
      which reduces to a diagonal matrix with the Jordan ``blocks''
      $J_1(\pm\sqrt{\pm\sigma_j})$ after suitable permutations and
      applying Lemma~\ref{thm:svdfour}.
  \end{enumerate}
\end{proof}



\begin{corollary}\label{thm:svdvecs}
  The (right) eigenvectors belonging to the eigenvalues
  $\sqrt{\sigma_j}$, $-\sqrt{\sigma_j}$, $i\sqrt{\sigma_j}$, and
  $-i\sqrt{\sigma_j}$ are
  \begin{equation*}
    \begin{bmatrix}
      u_j \\
      v_j \\
      \sqrt{\sigma_j} u_j \\
      \sqrt{\sigma_j} v_j
    \end{bmatrix},
    \qquad
    \begin{bmatrix}
      -u_j \\
      -v_j \\
      \sqrt{\sigma_j} u_j \\
      \sqrt{\sigma_j} v_j
    \end{bmatrix},
    \qquad
    \begin{bmatrix}
      -iu_j \\
      \phantom{-}iv_j \\
      \phantom{-}\sqrt{\sigma_j} u_j \\
      -\sqrt{\sigma_j} v_j
    \end{bmatrix},
    \quad\text{and}\quad
    \begin{bmatrix}
      \phantom{-}iu_j \\
      -iv_j \\
      \phantom{-}\sqrt{\sigma_j} u_j \\
      -\sqrt{\sigma_j} v_j
    \end{bmatrix},
  \end{equation*}
  respectively. Here $u_j = U e_j$ and $v_j = V e_j$.
\end{corollary}


As we will see in the next two sections, we can follow the same general
approach with the QSVD and the RSVD. That is, we start with the
``obvious'' transformation of the pencils, permute the resulting block
matrices, and invoke an analogue of Lemma~\ref{thm:svdfour}.



\section{The quotient singular value decomposition}\label{sec:qsvd}

The QSVD can be used to solve, for example, generalized eigenvalue
problems of the form \eqref{eq:sqQSVD}, generalized total least squares
problems, general-form Tikhonov regularization, least squares problems
with equality constraints, etc.; see, e.g., Van~Loan~\cite{VanL76} and
Bai~\cite{ZB92gsvd} for more information. The QSVD is defined the
theorem below and comes from Paige and Saunders~\cite{PS81}, but has the
blocks of the partitioned matrices permuted to be closer to a special
case of the RSVD from Section~\ref{sec:rsvd}.


\begin{theorem}[Quotient singular value decomposition]\label{thm:qsvd}
  Let $A \in \mathbb{C}^{p \times q}$ and $C \in \mathbb{C}^{n \times
  q}$; then there exist a nonsingular matrix $Y \in \mathbb{C}^{q \times
  q}$ and unitary matrices $U \in \mathbb{C}^{p \times p}$ and $V \in
  \mathbb{C}^{n \times n}$ such that
  \begin{equation*}
    U^*\!AY =\;
    \bordermatrix[{[]}]{%
        & q_1 & q_2 & q_3      & q_4 \cr
    p_1 & 0   & 0   & D_\alpha & 0 \cr
    p_2 & 0   & 0   & 0        & I \cr
    p_3 & 0   & 0   & 0        & 0
    }
    \quad\text{and}\quad
    V^*CY =\;
    \bordermatrix[{[]}]{%
        & q_1 & q_2 & q_3      & q_4 \cr
    n_1 & 0   & I   & 0        & 0 \cr
    n_2 & 0   & 0   & D_\gamma & 0 \cr
    n_3 & 0   & 0   & 0        & 0
    },
  \end{equation*}
  where $n_2 = p_1 = q_3$, $n_1 = q_2$, and $p_2 = q_4$. Furthermore,
  $D_\alpha = \diag(\alpha_1, \dots, \alpha_{p_1})$ and $D_\gamma =
  \diag(\gamma_1, \dots, \gamma_{p_1})$ are such that $\alpha_j,
  \gamma_j > 0$ and $\alpha_j^2 + \gamma_j^2 = 1$ for all $j = 1$, $2$,
  \dots, $p_1$. These $p_1 = q_3$ pairs $(\alpha_j, \gamma_j)$ together
  with $q_2$ pairs $(0,1)$ and $q_4$ pairs $(1,0)$ are called the
  nontrivial pairs. The remaining $q_1$ pairs $(0,0)$ are called the
  trivial pairs.  Each nontrivial pair $(\alpha, \gamma)$ corresponds to
  a quotient singular value $\sigma = \alpha / \gamma$, where the result
  is $\infty$ by convention if $\gamma = 0$.
\end{theorem}


As before, we first consider the one dimensional case and generalize
Lemma~\ref{thm:svdfour} to the QSVD.


\begin{lemma}\label{thm:qsvdfour}
  Suppose $\alpha$ and $\gamma$ are positive real numbers and consider
  the pencil
  \begin{equation*}
    \mathcal A - \lambda \mathcal B =
    \begin{bmatrix}
      0 & \alpha & 0 & 0 \\
      \alpha & 0 & 0 & 0 \\
      0 & 0 & 1 & 0 \\
      0 & 0 & 0 & 1
    \end{bmatrix}
    {} - \lambda
    \begin{bmatrix}
      0 & 0 & 1 & 0 \\
      0 & 0 & 0 & \gamma \\
      1 & 0 & 0 & 0 \\
      0 & \gamma & 0 & 0
    \end{bmatrix}.
  \end{equation*}
  Then the nonsingular matrices
  \begin{equation*}
    \mathcal X = \mathcal Y = \sigma^{-1/4} \diag(
      1, \gamma^{-1}, \sqrt\sigma, \sqrt\sigma),
  \end{equation*}
  where $\sigma = \alpha / \gamma$, are such that
  $\mathcal Y^* (\mathcal A - \lambda \mathcal B) \mathcal X$ is a
  pencil of the form \eqref{eq:svdfour}.
\end{lemma}
\begin{proof}
  The proof is by direct verification.
\end{proof}


Using the above definition of the QSVD in Theorem~\ref{thm:qsvd}, and
the reduction in Lemma~\ref{thm:qsvdfour}, we can state and prove the
following theorem and corollary.


\begin{theorem}\label{thm:sqfreeQSVD}
  Let $A$, $C$, and their corresponding QSVD be as in
  Theorem~\ref{thm:qsvd}. Then the KCF of the pencil
  \begin{equation}\label{eq:sqfreeQSVD}
    \mathcal A - \mathcal B =
    \begin{bmatrix}
      0 & A & 0 & 0 \\
      A^* & 0 & 0 & 0 \\
      0 & 0 & I & 0 \\
      0 & 0 & 0 & I
    \end{bmatrix}
    {} - \lambda
    \begin{bmatrix}
      0 & 0 & I & 0 \\
      0 & 0 & 0 & C^* \\
      I & 0 & 0 & 0 \\
      0 & C & 0 & 0
    \end{bmatrix}
  \end{equation}
  consists of the following blocks.
  \begin{enumerate}
    \item A $q_1 \times q_1$ zero block, which corresponds to quotient
      singular pairs of the form $(0, 0)$.
    \item A series of $n_3$ blocks $N_1$, which correspond to $(1,0)$
      pairs.
    \item A series of $p_2$ blocks $N_3$, which correspond to $(1,0)$
      pairs.
    \item A series of $p_3 + q_2$ blocks $J_2(0)$, which correspond to
      $(0, 1)$ pairs.
    \item The blocks
      $J_1(\sqrt{\sigma_1})$, \dots, $J_1(\sqrt{\sigma_{p_1}})$, 
      $J_1(-\sqrt{\sigma_1})$, \dots, $J_1(-\sqrt{\sigma_{p_1}})$,
      $J_1(i\sqrt{\sigma_1})$, \dots, $J_1(i\sqrt{\sigma_{p_1}})$,
      $J_1(-i\sqrt{\sigma_1})$, \dots, $J_1(-i\sqrt{\sigma_{p_1}})$,
      where $\sigma_1$, \dots, $\sigma_{p_1}$ are the finite and nonzero
      quotient singular values of the matrix pair $(A, C)$.
  \end{enumerate}
\end{theorem}
\begin{proof}
  Let $\mathcal A_0 - \lambda \mathcal B_0 = \mathcal A - \lambda
  \mathcal B$, and define the transformations $\mathcal X_0 = \mathcal
  Y_0 = \diag(U, Y, U, V)$. Then the pencil $\mathcal A_1 - \lambda
  \mathcal B_1 = \mathcal Y_0^* (\mathcal A_0 - \lambda \mathcal B_0)
  \mathcal X_0$ is a square $13 \times 13$ block matrix of dimension
  \begin{equation*}
    \underbrace{p_1 + p_2 + p_3}_{p}
    {} + \underbrace{q_1 + q_2 + q_3 + q_4}_{q}
    {} + \underbrace{p_1 + p_2 + p_3}_{p}
    {} + \underbrace{n_1 + n_2 + n_3}_{n}.
  \end{equation*}
  Now let $\mathcal X_1$ and $\mathcal Y_1$ be the permutation matrices
  corresponding to the permutations
  \begin{equation*}
    \begin{split}
      \pi_{\mathcal X} &= (
        4,\quad
        13,\quad
        7, 9, 2,\quad
        3, 10, 5, 11,\quad
        1, 6, 8, 12
      ),
      \\
      \pi_{\mathcal Y} &= (
        4,\quad
        13,\quad
        2, 9, 7,\quad
        10, 3, 11, 5,\quad
        1, 6, 8, 12
      ),
    \end{split}
  \end{equation*}
  respectively. Then the pencil $\mathcal A_2 - \lambda \mathcal B_2 =
  \mathcal Y_1^* (\mathcal A_1 - \lambda \mathcal B_1) \mathcal X_1$ is
  block diagonal, and has the following blocks along its diagonal.
  \begin{enumerate}
    \item A $q_1 \times q_1$ block of zeros.
    \item The $n_3 \times n_3$ block $I - \lambda \cdot 0 = N_1 \otimes
      I$.
    \item The $(p_2 + p_2 + q_4) \times (q_4 + p_2 + p_2)$ block
      \begin{equation*}
        \begin{bmatrix}
          I & 0 & 0 \\
          0 & I & 0 \\
          0 & 0 & I
        \end{bmatrix}
        {} - \lambda
        \begin{bmatrix}
          0 & I & 0 \\
          0 & 0 & I \\
          0 & 0 & 0
        \end{bmatrix}
        = N_3 \otimes I,
      \end{equation*}
      where $p_2 = q_4$.
    \item The $(p_3 + p_3 + q_2 + n_1) \times (p_3 + p_3 + n_1 + q_2)$
      block
      \begin{equation*}
        \begin{bmatrix}
          0 & I & 0 & 0 \\
          0 & 0 & 0 & 0 \\
          0 & 0 & 0 & I \\
          0 & 0 & 0 & 0
        \end{bmatrix}
        {} - \lambda
        \begin{bmatrix}
          I & 0 & 0 & 0 \\
          0 & I & 0 & 0 \\
          0 & 0 & I & 0 \\
          0 & 0 & 0 & I
        \end{bmatrix}
        =
        \begin{bmatrix}
          J_2(0) \otimes I \\
          & J_2(0) \otimes I
        \end{bmatrix},
      \end{equation*}
      where $n_1 = q_2$, which yields $p_3 + q_2$ blocks $J_2(0)$ in the
      KCF of \eqref{eq:sqfreeQSVD} after suitable permutations.

    \item The $(p_1 + q_3 + p_1 + n_2) \times (p_1 + q_3 + p_1 + n_2)$
      block
      \begin{equation*}
        \begin{bmatrix}
          0 & D_\alpha & 0 & 0 \\
          D_\alpha & 0 & 0 & 0 \\
          0 & 0 & I & 0 \\
          0 & 0 & 0 & I
        \end{bmatrix}
        {} - \lambda
        \begin{bmatrix}
          0 & 0 & I & 0 \\
          0 & 0 & 0 & D_\gamma \\
          I & 0 & 0 & 0 \\
          0 & D_\gamma & 0 & 0
        \end{bmatrix},
      \end{equation*}
      where $n_2 = p_1 = q_3$, which reduces to a diagonal matrix with
      the Jordan ``blocks'' $J_1(\pm\sqrt{\pm\sigma_j})$ after suitable
      permutations and applying Lemma~\ref{thm:qsvdfour}.
  \end{enumerate}
\end{proof}



\begin{corollary}\label{thm:qsvdvecs}
  The eigenvectors belonging to the nonzero finite eigenvalues
  $\sqrt{\sigma_j}$, $-\sqrt{\sigma_j}$, $i\sqrt{\sigma_j}$, and
  $-i\sqrt{\sigma_j}$ are
  \begin{equation*}
    \begin{bmatrix}
      u_j \\
      \gamma_j^{-1} y_j \\
      \sqrt{\sigma_j} u_j \\
      \sqrt{\sigma_j} v_j
    \end{bmatrix},
    \qquad
    \begin{bmatrix}
      -u_j \\
      -\gamma_j^{-1} y_j \\
      \sqrt{\sigma_j} u_j \\
      \sqrt{\sigma_j} v_j
    \end{bmatrix},
    \qquad
    \begin{bmatrix}
      -iu_j \\
      \phantom{-}i\gamma_j^{-1} y_j \\
      \phantom{-}\sqrt{\sigma_j} u_j \\
      -\sqrt{\sigma_j} v_j
    \end{bmatrix},
    \quad\text{and}\quad
    \begin{bmatrix}
      \phantom{-}iu_j \\
      -i\gamma_j^{-1} y_j \\
      \phantom{-}\sqrt{\sigma_j} u_j \\
      -\sqrt{\sigma_j} v_j
    \end{bmatrix},
  \end{equation*}
  respectively. Here $u_j = U e_j$, $v_j = V e_j$, and $y_j = Y e_{q_1 +
  j}$.
\end{corollary}


Compared to the OSVD, the KCF for \eqref{eq:sqfreeQSVD} has three extra
blocks, nameley the first three blocks in the list of
Theorem~\ref{thm:sqfreeQSVD}. The first of these blocks is associated
with the singularity of the pencil, while the second and third blocks
are associated with infinite eigenvalues/singular values.



\section{The restricted singular value decomposition}\label{sec:rsvd}

The RSVD is useful for, for example, analyzing structured rank
perturbations, computing low-rank approximations of partitioned
matrices, minimization or maximization of the bilinear form $x^*\!  A y$
under the constraints $\|B^* x\|, \|C y\| \neq 0$, solving matrix
equations of the form $BXC = A$, constrained total least squares with
exact rows and columns, and generalized Gauss--Markov models with
constraints. See, e.g., Zha~\cite{Zha91} and De~Moor and
Golub~\cite{MG91} for more information. The RSVD is defined by the
theorem below and comes from \cite{Zwaa19}, but was adapted from the
preceeding two references.  The full definition of the RSVD is quite
tedious, but we can follow the same general approach as in the previous
two sections.


\newlength{\DIndexWidth}
\settowidth{\DIndexWidth}{\hbox{$D_\alpha$}}
\newlength{\ZeroWidth}
\settowidth{\ZeroWidth}{\hbox{$0$}}

\begin{theorem}[The restricted singular value
  decomposition]\label{thm:rsvd}
  Let $A \in \mathbb C^{p \times q}$, $B \in \mathbb C^{p \times m}$,
  and $C \in \mathbb C^{n \times q}$, and define $r_{A} = \rank A$,
  $r_{B} = \rank B$, $r_{C} = \rank C$, $r_{AB} = \rank \big[ A\; B
  \big]$, $r_{AC} = \rank \big[ A;\; C \big]$, and $r_{ABC} = \rank
  \minimat{A}{B}{C}{0}$. Then the triplet of matrices $(A,B,C)$ can be
  factorized as $A = X^{-*} \Sigma_\alpha Y^{-1}$, $B = X^{-*}
  \Sigma_\beta U^*$, and $C = V \Sigma_\gamma Y^{-1}$, where $X \in
  \mathbb C^{p\times p}$ and $Y \in \mathbb C^{q\times q}$ are
  nonsingular, and $U \in \mathbb C^{m\times m}$ and $V \in \mathbb
  C^{n\times n}$ are orthonormal.  Furthermore, $\Sigma_\alpha$,
  $\Sigma_\beta$, and $\Sigma_\gamma$ have nonnegative entries and are
  such that
  {\scriptsize $\arraycolsep=0.25\arraycolsep
    \def\arraystretch{1}
    \left[\begin{array}{c|c}
      \Sigma_\alpha & \Sigma_\beta \\ \hline \Sigma_\gamma & 0
    \end{array}\right]$}
  can be written as
  \begin{equation}\label{eq:sigmas}
    \begin{array}{cccccccccccc}
      &
      \multicolumn{10}{c}{\hspace{-2.0\arraycolsep}
        \begin{array}{cccccccccc}
          \text{\makebox[\ZeroWidth]{\small $q_1$}} &
          \text{\makebox[\ZeroWidth]{\small $q_2$}} &
          \text{\makebox[\DIndexWidth]{\small $q_3$}} &
          \text{\makebox[\ZeroWidth]{\small $q_4$}} &
          \text{\makebox[\ZeroWidth]{\small $q_5$}} &
          \text{\makebox[\ZeroWidth]{\small $q_6$}} &
          \text{\makebox[\DIndexWidth]{\small $m_1$}} &
          \text{\makebox[\ZeroWidth]{\small $m_2$}} &
          \text{\makebox[\ZeroWidth]{\small $m_3$}} &
          \text{\makebox[\ZeroWidth]{\small $m_4$}}
        \end{array}
      }
      \\
      \begin{array}{c}
        \text{\small $p_1$} \\
        \text{\small $p_2$} \\
        \text{\small $p_3$} \\
        \text{\small $p_4$} \\
        \text{\small $p_5$} \\
        \text{\small $p_6$} \\
        \text{\small $n_1$} \\
        \text{\small $n_2$} \\
        \text{\small $n_3$} \\
        \text{\small $n_4$}
      \end{array}
      &
      \multicolumn{10}{c}{\hspace{-2.0\arraycolsep}\left[
        \begin{array}{cc:cccc|ccc:c}
          0 & 0 & D_\alpha & 0 & 0 & 0 & D_\beta & 0 & 0 & 0 \\
          0 & 0 & 0        & I & 0 & 0 &       0 & I & 0 & 0 \\
          0 & 0 & 0        & 0 & I & 0 &       0 & 0 & 0 & 0 \\
          0 & 0 & 0        & 0 & 0 & I &       0 & 0 & 0 & 0 \\
          \hdashline
          0 & 0 & 0        & 0 & 0 & 0 &       0 & 0 & 0 & I \\
          0 & 0 & 0        & 0 & 0 & 0 &       0 & 0 & 0 & 0 \\
          \hline
          0 & I & 0        & 0 & 0 & 0 \\\cdashline{1-6}
          0 & 0 & D_\gamma & 0 & 0 & 0 \\
          0 & 0 & 0        & 0 & I & 0 \\
          0 & 0 & 0        & 0 & 0 & 0 \\
        \end{array}
      \right]}
      &\hspace{-2.0\arraycolsep}
      \begin{array}{l}
        \text{\small $p_1 = q_3 = r_{ABC} + r_{A} - r_{AB} - r_{AC}$} \\
        \text{\small $p_2 = q_4 = r_{AC} + r_{B} - r_{ABC}$} \\
        \text{\small $p_3 = q_5 = r_{AB} + r_{C} - r_{ABC}$} \\
        \text{\small $p_4 = q_6 = r_{ABC} - r_{B} - r_{C}$} \\
        \text{\small $p_5 = r_{AB} - r_{A}$, $q_2 = r_{AC} - r_{A}$} \\
        \text{\small $p_6 = p - r_{AB}$, $q_1 = q - r_{AC}$} \\
        \text{\small $n_1 = q_2$, $m_4 = p_5$} \\
        \text{\small $n_2 = m_1 = p_1 = q_3$} \\
        \text{\small $n_3 = p_3 = q_5$, $m_2 = p_2 = q_4$} \\
        \text{\small $n_4 = n - r_{C}$, $m_3 = m - r_{B}$},
      \end{array}
    \end{array}
  \end{equation}
  where $D_\alpha = \diag(\alpha_1, \dots, \alpha_{p_1})$, $D_\beta =
  \diag(\beta_1, \dots, \beta_{p_1})$, and $D_\gamma = \diag(\gamma_1,
  \dots, \gamma_{p_1})$. Moreover, $\alpha_j$, $\beta_j$, and $\gamma_j$
  are scaled such that $\alpha_j^2 + \beta_j^2 \gamma_j^2 = 1$ for $i =
  1$, \dots, $p_1$. Besides the $p_1$ triplets $(\alpha_j,
  \beta_j, \gamma_j)$, there are $p_2$ triplets $(1,1,0)$, $p_3$
  triplets $(1,0,1)$, $p_4$ triplets $(1,0,0)$, and $\min \{ p_5, q_2
  \}$ triplets $(0,1,1)$.  This leads to a total of $p_1 + p_2 + p_3 +
  p_4 + \min \{ p_5, q_2 \} = r_A + \min \{ p_5, q_2 \} = \min \{
  r_{AB}, r_{AC} \}$ regular triplets of the form $(\alpha, \beta,
  \gamma)$ with $\alpha^2 + \beta^2 \gamma^2 = 1$. Each of these
  triplets corresponds to a \emph{restricted singular value} $\sigma =
  \alpha / (\beta \gamma)$, where the result is $\infty$ by convention
  if $\alpha \neq 0$ and $\beta \gamma = 0$.
  Finally, the triplet has a right (or column) trivial block of
  dimension $q_1 = \dim(\mathcal N(A) \cap \mathcal N(C))$, and
  a left (or row) trivial block of dimension $p_6 = \dim(\mathcal N(A^*)
  \cap \mathcal N(B^*))$.
  %
\end{theorem}


As before, we first consider the one dimensional case and generalize
Lemmas~\ref{thm:svdfour} and~\ref{thm:qsvdfour} to the RSVD.


\begin{lemma}\label{thm:rsvdfour}
  Suppose $\alpha$, $\beta$, and $\gamma$ are positive real numbers and
  consider the pencil
  \begin{equation*}
    \mathcal A - \lambda \mathcal B =
    \begin{bmatrix}
      0 & \alpha & 0 & 0 \\
      \alpha & 0 & 0 & 0 \\
      0 & 0 & 1 & 0 \\
      0 & 0 & 0 & 1
    \end{bmatrix}
    {} - \lambda
    \begin{bmatrix}
      0 & 0 & \beta & 0 \\
      0 & 0 & 0 & \gamma \\
      \beta & 0 & 0 & 0 \\
      0 & \gamma & 0 & 0
    \end{bmatrix}.
  \end{equation*}
  Then the nonsingular matrices
  \begin{equation*}
    \mathcal X = \mathcal Y = \sigma^{-1/4} \diag(
      \beta^{-1}, \gamma^{-1}, \sqrt\sigma, \sqrt\sigma),
  \end{equation*}
  where $\sigma = \alpha / (\beta\gamma)$,
  are such that
  $\mathcal Y^* (\mathcal A - \lambda \mathcal B) \mathcal X$ is a
  pencil of the form \eqref{eq:svdfour}.
\end{lemma}
\begin{proof}
  The proof is by direct verification.
\end{proof}


With the definition of the RSVD in Theorem~\ref{thm:rsvd}, and the
reduction in Lemma~\ref{thm:rsvdfour}, we can state and prove the
following theorem and corollary.


\begin{theorem}\label{thm:sqfreeRSVD}
  Let $A$, $B$, $C$, and their corresponding RSVD be as in
  Theorem~\ref{thm:rsvd}. Then the KCF of the pencil
  \begin{equation}\label{eq:sqfreeRSVD}
    \mathcal A - \lambda \mathcal B =
    \begin{bmatrix}
      0 & A & 0 & 0 \\
      A^* & 0 & 0 & 0 \\
      0 & 0 & I & 0 \\
      0 & 0 & 0 & I
    \end{bmatrix}
    {} - \lambda
    \begin{bmatrix}
      0 & 0 & B & 0 \\
      0 & 0 & 0 & C^* \\
      B^* & 0 & 0 & 0 \\
      0 & C & 0 & 0
    \end{bmatrix}
  \end{equation}
  consists of the following blocks.
  \begin{enumerate}
    \item A $(p_6 + q_1) \times (p_6 + q_1)$ zero block, which
      correspond to restricted singular triplets of the form $(0,0,0)$.
    \item A series of $p_4 + q_6 + m_3 + n_4$ blocks $N_1$, which
      correspond to $(1,0,0)$ triplets.
    \item A series of $p_2$ blocks $N_3$, which correspond to
      $(1,1,0)$ triplets.
    \item A series of $p_3$ blocks $N_3$, which correspond to
      $(1,0,1)$ triplets.
    \item A series of $p_5 + q_2 = m_4 + n_1$ blocks $J_2(0)$, which
      correspond to $(0,1,1)$ triplets.
    \item The blocks
      $J_1(\sqrt{\sigma_1})$, \dots, $J_1(\sqrt{\sigma_{p_1}})$, 
      $J_1(-\sqrt{\sigma_1})$, \dots, $J_1(-\sqrt{\sigma_{p_1}})$,
      $J_1(i\sqrt{\sigma_1})$, \dots, $J_1(i\sqrt{\sigma_{p_1}})$,
      $J_1(-i\sqrt{\sigma_1})$, \dots, $J_1(-i\sqrt{\sigma_{p_1}})$,
      where $\sigma_1$, \dots, $\sigma_{p_1}$ are the finite and nonzero
      restricted singular values of the matrix triplet $(A, B, C)$.
  \end{enumerate}
\end{theorem}
\begin{proof}
  Let $\mathcal A_{0} - \lambda \mathcal B_{0} = \mathcal A -
  \lambda \mathcal B$, and define the transformations $\mathcal X_{0} =
  \mathcal Y_0 = \operatorname{diag}(X, Y, U, V)$.  Then the pencil
  $\mathcal A_1 - \lambda \mathcal B_1 = \mathcal Y_0^* (\mathcal A_0 -
  \lambda \mathcal B_0) \mathcal X_0$ is a square $20 \times 20$ block
  matrix of dimension
  \begin{equation*}
    \underbrace{p_1 + \dots + p_6}_{p}
    {} + \underbrace{q_1 + \dots + q_6}_{q}
    {} + \underbrace{m_1 + \dots +m_4}_{m}
    {} + \underbrace{n_1 + \dots + n_4}_{n}.
  \end{equation*}
  Now let $\mathcal X_1$ and $\mathcal Y_1$ be permutation matrices
  corresponding to the permutations
  \begin{equation*}
    \begin{split}
      \pi_{\mathcal X} &= (
        6, 7,\quad 
        12, 4, 15, 20, \quad 
        10, 14, 2, \quad 
        3, 19, 11, \quad 
        5, 16, 8, 17, \quad 
        1, 9, 13, 18 
      ), \\
      \pi_{\mathcal Y} &= (
        6, 7, \quad 
        4, 12, 15, 20, \quad 
        2, 14, 10, \quad 
        11, 19, 3, \quad 
        16, 5, 17, 8, \quad 
        1, 9, 13, 18 
      ),
    \end{split}
  \end{equation*}
  respectively.  Then the pencil $\mathcal A_2 - \lambda \mathcal B_2 =
  \mathcal Y_1^* (\mathcal A_1 - \lambda \mathcal B_1) \mathcal X_1$ is
  block diagonal, and has the following blocks along its diagonal.
  \begin{enumerate}
    \item A $(p_6 + q_1) \times (p_6 + q_1)$ block of zeros.
    \item The $(p_4 + q_6 + m_3 + n_4) \times (q_6 + p_4 + m_3 + n_4)$
      block
      \begin{equation*}
        \begin{bmatrix}
          I & 0 & 0 & 0 \\
          0 & I & 0 & 0 \\
          0 & 0 & I & 0 \\
          0 & 0 & 0 & I
        \end{bmatrix}
        {} - \lambda
        \begin{bmatrix}
          0 & 0 & 0 & 0 \\
          0 & 0 & 0 & 0 \\
          0 & 0 & 0 & 0 \\
          0 & 0 & 0 & 0
        \end{bmatrix}
        = N_1 \otimes I,
      \end{equation*}
      where $p_4 = q_6$.

    \item The $(p_2 + m_2 + q_4) \times (q_4 + m_2 + p_2)$ block
      \begin{equation*}
        \begin{bmatrix}
          I & 0 & 0 \\
          0 & I & 0 \\
          0 & 0 & I
        \end{bmatrix}
        {} - \lambda
        \begin{bmatrix}
          0 & I & 0 \\
          0 & 0 & I \\
          0 & 0 & 0
        \end{bmatrix}
        = N_3 \otimes I,
      \end{equation*}
      where $m_2 = p_2 = q_4$.

    \item The $(q_5 + n_3 + p_3) \times (p_3 + n_3 + q_5)$ block
      \begin{equation*}
        \begin{bmatrix}
          I & 0 & 0 \\
          0 & I & 0 \\
          0 & 0 & I
        \end{bmatrix}
        {} - \lambda
        \begin{bmatrix}
          0 & I & 0 \\
          0 & 0 & I \\
          0 & 0 & 0
        \end{bmatrix}
        = N_3 \otimes I,
      \end{equation*}
      where $n_3 = p_3 = q_5$.

    \item The $(m_4 + p_5 + n_1 + q_2) \times (p_5 + m_4 + q_2 + n_1)$
      block
      \begin{equation*}
        \begin{bmatrix}
          0 & I & 0 & 0 \\
          0 & 0 & 0 & 0 \\
          0 & 0 & 0 & I \\
          0 & 0 & 0 & 0
        \end{bmatrix}
        {} - \lambda
        \begin{bmatrix}
          I & 0 & 0 & 0 \\
          0 & I & 0 & 0 \\
          0 & 0 & I & 0 \\
          0 & 0 & 0 & I
        \end{bmatrix}
        =
        \begin{bmatrix}
          J_2(0) \otimes I \\
          & J_2(0) \otimes I
        \end{bmatrix}
      \end{equation*}
      where $m_4 = p_5$ and $n_1 = q_2$, which yields $p_5 + q_2$ blocks
      $J_2(0)$ in the KCF of \eqref{eq:sqfreeRSVD} after suitable
      permutations.

    \item
      The $(p_1 + q_3 + m_1 + n_2) \times (p_1 + q_3 + m_1 + n_2)$ block
      \begin{equation*}
        \begin{bmatrix}
          0 & D_\alpha & 0 & 0 \\
          D_\alpha & 0 & 0 & 0 \\
          0 & 0 & I & 0 \\
          0 & 0 & 0 & I
        \end{bmatrix}
        {} - \lambda
        \begin{bmatrix}
          0 & 0 & D_\beta & 0 \\
          0 & 0 & 0 & D_\gamma \\
          D_\beta & 0 & 0 & 0 \\
          0 & D_\gamma & 0 & 0
        \end{bmatrix},
      \end{equation*}
      where $m_1 = n_2 = p_1 = q_3$, which reduces to a diagonal matrix
      with the Jordan ``blocks'' $J_1(\pm\sqrt{\pm\sigma_j})$ after
      suitable permutations and applying Lemma~\ref{thm:rsvdfour}.
  \end{enumerate}
\end{proof}



\begin{corollary}\label{thm:rsvdvecs}
  The eigenvectors belonging to the nonzero finite eigenvalues
  $\sqrt{\sigma_j}$, $-\sqrt{\sigma_j}$, $i\sqrt{\sigma_j}$, and
  $-i\sqrt{\sigma_j}$ are
  \begin{equation*}
    \begin{bmatrix}
      \beta_j^{-1} x_j \\
      \gamma_j^{-1} y_j \\
      \sqrt{\sigma_j} u_j \\
      \sqrt{\sigma_j} v_j
    \end{bmatrix},
    \qquad
    \begin{bmatrix}
      -\beta_j^{-1} x_j \\
      -\gamma_j^{-1} y_j \\
      \sqrt{\sigma_j} u_j \\
      \sqrt{\sigma_j} v_j
    \end{bmatrix},
    \qquad
    \begin{bmatrix}
      -i\beta_j^{-1} x_j \\
      \phantom{-}i\gamma_j^{-1} y_j \\
      \phantom{-}\sqrt{\sigma_j} u_j \\
      -\sqrt{\sigma_j} v_j
    \end{bmatrix},
    \quad\text{and}\quad
    \begin{bmatrix}
      \phantom{-}i\beta_j^{-1} x_j \\
      -i\gamma_j^{-1} y_j \\
      \phantom{-}\sqrt{\sigma_j} u_j \\
      -\sqrt{\sigma_j} v_j
    \end{bmatrix},
  \end{equation*}
  respectively. Here $u_j = U e_j$, $v_j = V e_{n_1 + j}$, $x_j = X
  e_j$, and $y_j = Y e_{q_1 + q_2 + j}$.
\end{corollary}


A comparison with the KCF of the typical augmented pencil
\eqref{eq:augRSVD} might be insightful. Using the transform $\diag(U,
Y)$ and noting that we get a $12 \times 12$ block matrix with sizes
\begin{equation*}
  p + q = p_1 + \dots + p_6 + q_1 + \dots + q_6,
\end{equation*}
we can use the permutations
\begin{equation*}
  \begin{split}
    \pi_{\mathcal X} &= (
      6, 7, \quad
      12, 4, \quad
      10, 2, \quad
      3, 11, \quad
      5, 8, \quad
      1, 9
    ),
    \\
    \pi_{\mathcal Y} &= (
      6, 7, \quad
      4, 12, \quad
      2, 10, \quad
      11, 3, \quad
      5, 8, \quad
      1, 9
    ).
  \end{split}
\end{equation*}
to see that we get the following blocks for the KCF of
\eqref{eq:augRSVD}.
\begin{enumerate}
  \item A $(p_6 + q_1) \times (p_6 + q_1)$ zero block.
  \item A $(p_4 + q_6) \times (q_6 + p_4)$ block $N_1 \otimes I$.
  \item A $(p_2 + q_4) \times (q_4 + p_2)$ block $N_2 \otimes I$.
  \item A $(p_3 + q_5) \times (q_5 + p_3)$ block $N_2 \otimes I$.
  \item A $(p_5 + q_2) \times (p_5 + q_2)$ block $J_1(0) \otimes I$.
  \item A $(p_1 + q_3) \times (p_1 + q_3)$ block
    \begin{equation*}
      \begin{bmatrix}
        0 & D_\alpha \\
        D_\alpha & 0
      \end{bmatrix}
      {} - \lambda
      \begin{bmatrix}
        D_\beta^2 & 0 \\
        0 & D_\gamma^2
      \end{bmatrix}.
    \end{equation*}
\end{enumerate}
Hence, we see that the KCF of \eqref{eq:augRSVD} and
\eqref{eq:sqfreeRSVD} have comparable blocks.  The primary qualitative
difference is that the zero singular values do not result in Jordan
blocks of size larger than 1.



\section{A tree of generalized singular value decompositions}

For the OSVD we started with the matrix pencil \eqref{eq:sqfreeSVD},
where $\mathcal A$ and $\mathcal B$ together have a total of eight
nonzero blocks. Two of these blocks are the considered matrix $A$ and
its Hermitian transpose $A^*$, the remaining six blocks are identity
matrices.  Then, for the QSVD we replaced two of these identity matrices
with the considered matrix $C$ and its Hermitian transpose $C^*$.  Next,
for the RSVD we replaced two more identity matrices with the considered
$B$ and its Hermitian transpose $B^*$. We now have two identity matrices
left, and a natural question to ask is: if we replace these two
identities with two other matrices, can we meaningfully interpret the
resulting pencil as a generalized SVD?  The answer to this question is
yes if we take
\begin{equation}\label{eq:qqqqsvd}
  \mathcal A - \lambda \mathcal B =
  \begin{bmatrix}
    0 & A & 0 & 0 \\
    A^* & 0 & 0 & 0 \\
    0 & 0 & D^*D & 0 \\
    0 & 0 & 0 & EE^*
  \end{bmatrix}
  {} - \lambda
  \begin{bmatrix}
    0 & 0 & B & 0 \\
    0 & 0 & 0 & C^* \\
    B^* & 0 & 0 & 0 \\
    0 & C & 0 & 0
  \end{bmatrix},
\end{equation}
where $D \in \mathbb{C}^{k \times m}$ and $E \in \mathbb{C}^{n \times
l}$ for some $k,l \ge 1$. To see this, suppose for simplicity that $D$
and $E$ are both invertible, then the pencil above corresponds to the
RSVD of the triplet $(A, BD^{-1}, E^{-1}C)$. In turn, the restricted
singular values of the latter triplet coincide with the ordinary
singular values of the product $DB^{-1}\!AC^{-1}E$ if $B$ and $C$ are
also invertible.  The usefulness of this \textit{QQQQ-SVD} is unclear,
and the pencil \eqref{eq:qqqqsvd}, which is no longer cross
product--free, is primarily of theoretical interest. For the full
definition of this decomposition, see De~Moor and Zha~\cite{MZ91} for
more details.



\section{Numerical experiments}\label{sec:ne}

The general idea for the experiments in this section is straightforward.
Simply generate matrices with sufficiently large condition numbers, form
the pencils \eqref{eq:augQSVD}, \eqref{eq:sqfreeQSVD},
\eqref{eq:augRSVD}, and \eqref{eq:sqfreeRSVD}, and check the accuracy of
the computed generalized eigenvalues. The following table summarizes the
relation between the various singular value problems, their
reformulations as matrix pencils, and the corresponding eigenvalues.
\begin{center}
  \begin{tabular}{@{}rccc@{}}
    \toprule
    $\lambda$ & OSVD & QSVD & RSVD \\
    \midrule
    $\sigma^2$ & \eqref{eq:sqSVD} & \eqref{eq:sqQSVD} & --- \\
    $\pm\sigma$ & \eqref{eq:augSVD} & \eqref{eq:augQSVD} & \eqref{eq:augRSVD} \\
    $\pm \sqrt{\pm\sigma}$ & \eqref{eq:sqfreeSVD} & \eqref{eq:sqfreeQSVD} & \eqref{eq:sqfreeRSVD}\\
    \bottomrule
  \end{tabular}
\end{center}

We can generate the matrices in different ways, but the straightforward
approaches described next suffice to prove the usefulness of the new
pencils.  We proceed as follows for the QSVD; given the dimension $n$
and sufficiently large condition numbers $\kappa_Y$ and $\kappa_\Sigma$:
\begin{enumerate}
  \item Generate random orthonormal matrices $U_Y$, $V_Y$, $U$ and $V$
    \cite{Mezz07}.
  \item Compute $\Sigma_Y = \diag(\eta_1, \dots, \eta_n)$ and $Y = U_Y
    \Sigma_Y V_Y^T$, where $\eta_j = \kappa_Y^{1/2 - (j - 1) / (n -
    1)}$.
  \item Let $\Sigma = \diag(\sigma_1, \dots, \sigma_n)$, $\Sigma_\alpha
    = \diag(\alpha_1, \dots, \alpha_n)$, and $\Sigma_\gamma =
    \diag(\gamma_1, \dots, \gamma_n)$, where $\sigma_j
    = \kappa_\Sigma^{1/2 - (j - 1) / (n - 1)}$, $\alpha_j = \sigma_j
    (1 + \sigma_j^2)^{-1/2}$, and $\gamma_j = (1 + \sigma_j^2)^{-1/2}$.
  \item Compute $A = U \Sigma_\alpha Y^{-1}$ and $C = V \Sigma_\gamma
    Y^{-1}$.
\end{enumerate}
With the steps above we generate matrices satisfying $\kappa(Y) =
\kappa_Y$, $\Sigma = \Sigma_\alpha \Sigma_\gamma^{-1}$, $\Sigma_\alpha^2
+ \Sigma_\gamma^2 = I$, $\kappa(\Sigma) = \kappa_\Sigma$, and
$\kappa(\Sigma_\alpha) = \kappa(\Sigma_\gamma) = \kappa_\Sigma^{1/2}$.
To ensure we do not lose precision prematurely, we must do all the
computations in higher-precision arithmetic\footnote{The required random
numbers are generated in double precision for simplicity.} and convert
$A$ and $C$ to IEEE 754 double precision at the end.  Given the
dimension $n$ and the condition numbers $\kappa_X$, $\kappa_Y$, and
$\kappa_\Sigma$, we can take a similar approach for the RSVD.
\begin{enumerate}
  \item Generate $U$, $V$, $Y$, $\Sigma$, $\Sigma_\alpha$, and
    $\Sigma_\gamma$ as before, and let $\Sigma_\beta = I$.
  \item Generate $X$ in the same way as $Y$, but with the condition
    number $\kappa_X$ instead of $\kappa_Y$.
  \item Let $A = X^{-T} \Sigma_\alpha Y^{-1}$, $B = X^{-T}
    \Sigma_\beta U^T$, and $C = V \Sigma_\gamma Y^{-1}$.
\end{enumerate}
Again, we have to do all the computations in higher-precision arithmetic
and convert $A$, $B$, and $C$ to double precision at the end.

Since the generated $B$ and $C$ have full row and column ranks,
respectively, the pencils \eqref{eq:augQSVD} and \eqref{eq:augRSVD} are
Hermitian positive definite generalized eigenvalue problems. This means
that we can use specialized solvers to compute the generalized
eigenvalues and vectors. The same is not true for \eqref{eq:sqfreeQSVD}
and \eqref{eq:sqfreeRSVD}, where both $\mathcal A$ and $\mathcal B$ are
indefinite. Thus, we have to use generic eigensolvers that do not take
the special structure of the pencils into account. And as a result, the
computed eigenvalues may not be purely real or purely imaginary.
Furthermore, suppose that
\begin{equation*}
  \widetilde \lambda_1 \approx \sqrt{\sigma},
  \qquad
  \widetilde \lambda_2 \approx i\sqrt{\sigma},
  \qquad
  \widetilde \lambda_3 \approx -\sqrt{\sigma},
  \quad\text{and}\quad
  \widetilde \lambda_4 \approx -i\sqrt{\sigma},
\end{equation*}
then we have no guarantee that the magnitudes $|\lambda_j|$ match to
high accuracy. The latter problem leads to the question: which
$|\lambda_j|$ should we pick? Anecdotal observations suggest that the
squared (absolute) geometric mean
\begin{equation}\label{eq:mean}
  (|\lambda_1|\, |\lambda_2|\, |\lambda_3|\, |\lambda_4|)^{-1/2}
  \approx \sigma
\end{equation}
is a reasonable choice.

We choose one small example to illustrate the accuracy lost from working
with cross products, and also the approximation picking problem. Suppose
that $n = 4$, $\kappa_Y = 10^{7}$, and $\kappa_\Sigma = 10$; then we get
the following results for a randomly generated matrix pair. First the
exact quotient singular values rounded to 12 digits after the point:
{
\begin{align*}
  3&.162277660168 &
  1&.467799267622 &
  0&.681292069058 &
  0&.316227766017.
\intertext{\normalsize Taking the square roots of the generalized
  eigenvalues of the pencil \eqref{eq:sqQSVD} gives:}
  3&.\underline{1}86032382196 &
  1&.\underline{467}953130046 &
  0&.\underline{681}383277148 &
  0&.\underline{316227}883689.
\intertext{\normalsize The magnitudes of the generalized eigenvalues of
  the augmented pencil \eqref{eq:augQSVD} are:}
  3&.\underline{1}86055633628 &
  1&.\underline{467}953295365 &
  0&.\underline{681}384866293 &
  0&.\underline{3162278}14411 \\
  3&.\underline{1}86055633628 &
  1&.\underline{467}953295365 &
  0&.\underline{681}384866293 &
  0&.\underline{3162278}14411.
\intertext{\normalsize The cross product--free pencil gives the squared
  magnitudes:}
  3&.\underline{162277}324135 &
  1&.\underline{46779926}3849 &
  0&.\underline{68129206}6129 &
  0&.\underline{3162277660}10 \\
  3&.\underline{16227766}1974 &
  1&.\underline{46779926761}8 &
  0&.\underline{681292069}112 &
  0&.\underline{31622776602}0 \\
  3&.\underline{16227766}1974 &
  1&.\underline{46779926761}8 &
  0&.\underline{681292069}112 &
  0&.\underline{31622776602}0 \\
  3&.\underline{162278}000456 &
  1&.\underline{46779927}1389 &
  0&.\underline{68129207}2105 &
  0&.\underline{3162277660}30.
\intertext{\normalsize Taking the squared absolute geometric means
  \eqref{eq:mean} yields:}
  3&.\underline{16227766}2135 &
  1&.\underline{46779926761}9 &
  0&.\underline{681292069}115 &
  0&.\underline{31622776602}0.
\intertext{\normalsize For comparison, the generalized singular values
  computed with LAPACK, which uses a Jacobi-type iteration (see, e.g.,
  Bai and Demmel~\cite{BD93} and the routine \texttt{xTGSJA}) are:}
  3&.\underline{162277659}936 &
  1&.\underline{467799267}555 &
  0&.\underline{6812920690}45 &
  0&.\underline{3162277660}24.
\end{align*}
}
The underlined digits of the approximations serve as a visual indicator
of their accuracy; if $d$ decimals places are underlined, then $d$ is
the largest integer for which the absolute error is less than $5 \cdot
10^{-(d+1)}$. The loss of accuracy caused by the cross products in the
pencils \eqref{eq:sqQSVD} and \eqref{eq:augQSVD} is obvious in the
results. We can also see that the digits of pairs of eigenvalues of the
typical augmented pencil \eqref{eq:augQSVD} match, while the final two
to four digits of the eigenvalues from the new cross product--free
pencil \eqref{eq:sqfreeQSVD} differ.  Still, the results of the latter
are about as accurate as the singular values computed by LAPACK
\cite{laug}.  Hence, users should prefer LAPACK for computing the QSVD
of dense matrix pairs, but may consider using \eqref{eq:sqfreeQSVD} in
combination with existing solvers for sparse matrix pairs.


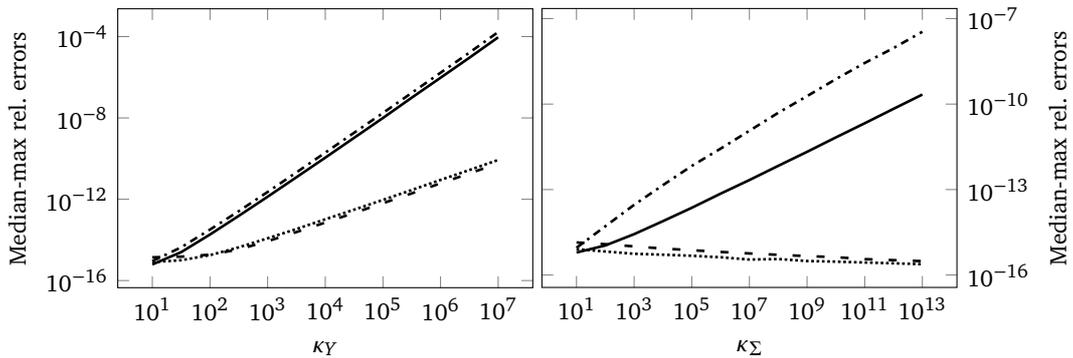
\begin{figure}[!htbp]
  \centering
  \begin{tikzpicture}
    \begin{loglogaxis}[
          height=0.33\textwidth
        , width=0.44\textwidth
        , xlabel={$\kappa_Y$}
        , ylabel={Median-max rel.\ errors}
        , xtick = {1e+1, 1e+2, 1e+3, 1e+4, 1e+5, 1e+6, 1e+7}
        , xminorticks = false
        ]

        \addplot+ table [x expr=10^\thisrow{k},y=aug] {\qsvdresultsfirst};
        \addplot+ table [x expr=10^\thisrow{k},y=new] {\qsvdresultsfirst};
        \addplot+ table [x expr=10^\thisrow{k},y=gsvd] {\qsvdresultsfirst};
        \addplot+ table [x expr=10^\thisrow{k},y=sq] {\qsvdresultsfirst};
    \end{loglogaxis}
  \end{tikzpicture}
  \begin{tikzpicture}
    \begin{loglogaxis}[
          height=0.33\textwidth
        , width=0.44\textwidth
        , xlabel={$\kappa_\Sigma$}
        , ylabel={Median-max rel.\ errors}
        , yticklabel pos=right
        , xtick = {1e+1, 1e+3, 1e+5, 1e+7, 1e+9, 1e+11, 1e+13}
        ]

        \addplot+ table [x expr=10^\thisrow{k},y=aug] {\qsvdresultssecond};
        \addplot+ table [x expr=10^\thisrow{k},y=new] {\qsvdresultssecond};
        \addplot+ table [x expr=10^\thisrow{k},y=gsvd] {\qsvdresultssecond};
        \addplot+ table [x expr=10^\thisrow{k},y=sq] {\qsvdresultssecond};
    \end{loglogaxis}
  \end{tikzpicture}
  \caption{The median of the maximum errors in the computed quotient
    singular values. On the left $\kappa_\Sigma = \kappa_X = 10$ and
    $\kappa_Y$ varies, and on the right $\kappa_X = \kappa_Y = 10$ and
    $\kappa_\Sigma$ varies. The dash-dotted line corresponds to the
    quadratic pencil \eqref{eq:sqQSVD}, the solid line to the typical
    augmented pencil \eqref{eq:augQSVD}, the dashed line to the cross
    product--free pencil \eqref{eq:sqfreeQSVD}, and the dotted line to
    the LAPACK results.}\label{fig:qsvd}
\end{figure}


For a more quantitative analysis, we can generate a large number of
matrix pairs and triplets for each combination of $n$, $\kappa_X$,
$\kappa_Y$, and $\kappa_\Sigma$. For each sample we compute the singular
values using the different methods we have available, and use the
(squared) absolute geometric mean, when necessary, to average the
approximations of each singular value.  Then we compute the
approximation errors and pick the maximum error for each matrix pair or
triplet. And finally, we take the median of all the maximum errors. As a
measure for the error we can use the chordal metric
\begin{equation*}
  \chi(\sigma, \widetilde \sigma)
  = \frac{|\sigma - \widetilde \sigma|}{%
      \sqrt{1 + \sigma^2}\; \sqrt{1 + \widetilde \sigma^2}}
  = \frac{|\sigma^{-1} - \widetilde \sigma^{-1}|}{%
      \sqrt{1 + \sigma^{-2}}\; \sqrt{1 + \widetilde \sigma^{-2}} }.
\end{equation*}
Here $\sigma$ is an exact (computed with high-precision arithmetic)
singular value, and $\widetilde \sigma$ is a computed approximation.


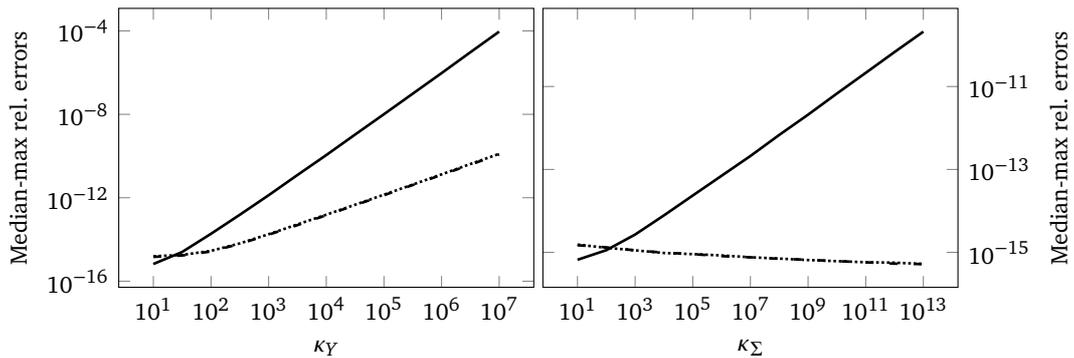
\begin{figure}[!htbp]
  \centering
  \begin{tikzpicture}
    \begin{loglogaxis}[
          height=0.33\textwidth
        , width=0.44\textwidth
        , xlabel={$\kappa_Y$}
        , ylabel={Median-max rel.\ errors}
        , xtick = {1e+1, 1e+2, 1e+3, 1e+4, 1e+5, 1e+6, 1e+7}
        , xminorticks = false
        ]

        \addplot+ table [x expr=10^\thisrow{k},y=aug] {\rsvdresultsfirst};
        \addplot+ table [x expr=10^\thisrow{k},y=new] {\rsvdresultsfirst};
        \addplot+ table [x expr=10^\thisrow{k},y=rsvd] {\rsvdresultsfirst};
    \end{loglogaxis}
  \end{tikzpicture}
  \begin{tikzpicture}
    \begin{loglogaxis}[
          height=0.33\textwidth
        , width=0.44\textwidth
        , xlabel={$\kappa_\Sigma$}
        , ylabel={Median-max rel.\ errors}
        , yticklabel pos=right
        , xtick = {1e+1, 1e+3, 1e+5, 1e+7, 1e+9, 1e+11, 1e+13}
        ]

        \addplot+ table [x expr=10^\thisrow{k},y=aug] {\rsvdresultssecond};
        \addplot+ table [x expr=10^\thisrow{k},y=new] {\rsvdresultssecond};
        \addplot+ table [x expr=10^\thisrow{k},y=rsvd] {\rsvdresultssecond};
    \end{loglogaxis}
  \end{tikzpicture}
  \caption{The median of the maximum errors in the computed restricted
    singular values. On the left $\kappa_X = \kappa_\Sigma = 10$ and
    $\kappa_Y$ varies, on the right $\kappa_\Sigma = 10$ and $\kappa_X =
    \kappa_Y$ varies; $n = 10$ in both cases. The solid line corresponds
    to the typical augmented pencil \eqref{eq:augRSVD}, the dashed line
    to the cross product--free pencil \eqref{eq:sqfreeRSVD}, and the
    dotted line (on top of the dashed line) to the results from the RSVD
    algorithm described in \cite{Zwaa19}.}\label{fig:rsvd}
\end{figure}


Figure~\ref{fig:qsvd} shows the results for the QSVD with 10000 samples
for each combination of parameters. The figure shows that we lose about
twice as much accuracy with \eqref{eq:sqQSVD} and \eqref{eq:augQSVD}
than with the new cross product--free pencil.  Furthermore, we see that
the quotient singular values computed from the new pencil are almost as
accurate as the ones computed with LAPACK. When we increase
$\kappa_\Sigma$ while keeping $\kappa_Y$ modest, we see that the squared
and augmented pencils lost accuracy as $\kappa_\Sigma$ increases, while
the results from the new pencil and LAPACK remain small.

Likewise, Figure~\ref{fig:rsvd} shows the results for the RSVD. Again we
see that we lose accuracy about twice as fast when using cross products
if $\kappa_Y$ increases and $\kappa_X$ and $\kappa_\Sigma$ remain
modest. Furthermore, we also again see that we lose accuracy with cross
products when $\kappa_X$ and $\kappa_Y$ remain modest and
$\kappa_\Sigma$ increases, while the cross product--free approaches keep
producing accurate results.



\section{Conclusion}\label{sec:con}

We have seen how we can reformulate the quotient and restricted singular
value problems as generalized eigenvalue problems, and without using
cross products like $A^*\!A$, $BB^*$, or $C^*C$. Moreover, the numerical
examples show the benefits of working with the cross product--free
pencils instead of the typical augmented pencils. That is, singular
values computed from the former may be more accurate than singular
values computed form the latter when $B$ or $C$ are ill-conditioned.

Still, when we use the cross product--free pencils we have to contend
with downsides that we have ignored so far. For example, for the QSVD
and square $A$ and $C$, the dimension of the cross product--free pencil
\eqref{eq:sqfreeQSVD} is twice as large as the dimension of the typical
augmented pencil \eqref{eq:augQSVD}, and four times as large as the
dimension of the squared pencil \eqref{eq:sqQSVD}. This makes a large
computational difference for dense solvers with cubic complexity.
Another problem is that we can no longer use specialized solvers for
Hermitian-positive-definite problems, even when $B$ is of full row rank
and $C$ of full column rank. We also have to contend with extra Jordan
blocks, for example for $\sigma = 0$, and the difficulties that
numerical software has when dealing with Jordan blocks.

But the above downsides of the cross product--free pencils also provide
opportunities for future research. For example, can we reduce the
performance overhead by developing solvers (e.g., for large-and-sparse
matrices) that exploit the block structure of the pencils? Can we
incorporate our knowledge of the structure of the spectrum to improve
the computed results?



\section*{Acknowledgements}

I wish to thank Andreas Frommer for his comments and suggestions.


\bibliographystyle{siamplain}
\bibliography{squarefree}

\end{document}